\documentclass[12pt,psamsfonts]{amsart}
\usepackage{amsmath}
\usepackage{amsthm}
\usepackage{amssymb}
\usepackage{amscd}
\usepackage{amsfonts}
\usepackage{amsbsy}

\def\be{\begin{equation}}
\def\ee{\end{equation}}

\def\bq{\begin{eqnarray}}
\def\eq{\end{eqnarray}}

\def\beq{\begin{eqnarray*}}
\def\eeq{\end{eqnarray*}}

\newtheorem {theorem} {Theorem}
\newtheorem {proposition} [theorem]{Proposition}
\newtheorem {corollary} [theorem]{Corollary}

\begin{document}

\title[Quadratic vector fields with invariant algebraic curve of large degree]
{Quadratic vector fields \\with invariant algebraic curve of large
degree}

\author[ R. Ramirez and N. Sadovskaia] { Rafael Ramirez$^1$ and Natalia Sadovskaia$^2$}

\address{$^1$ Departament d'Enginyeria Inform\`{a}tica i Matem\`{a}tiques, Universitat Rovira i Virgili, Avinguda dels Pa\"{\i}sos Catalans 26, 43007 Tarragona, Catalonia, Spain.}
\email{rafaelorlando.ramirez@urv.cat}

\address{$^2$ Departament de Matem\`{a}tica Aplicada II, Universitat Polit\`{e}cnica de Catalunya, C. Pau Gargallo 5, 08028 Barcelona, Catalonia, Spain.}
\email{natalia.sadovskaia@upc.edu}


\date{}
\dedicatory{}

\maketitle

\begin{abstract}
We construct a polynomial planar vector field of degree two with
one invariant algebraic curves  of large degree. We exhibit an
explicit quadratic vector fields which invariant curves of degree
nine , twelve , fifteen and eighteen degree.
 .
\end{abstract}

\section{Introduction}

In the paper \cite{Ll1} the authors present  for the first time
examples of algebraic limit cycles of degree greater than 4 for
planar quadratic vector fields. They also give an example of an
invariant algebraic curve of degree 12 and genus one for which the
quadratic system has no Darboux integrating factors or first
integrals \cite{Ll2}.

One of the first example of quadratic planar system with invariant
algebraic curve of genus one is the Filipstov differential system
\cite{RS}
\[
\left\{%
\begin{array}{cc}
\dot{x}=16(1+a)x-6(2+a)x^2+(2+12x)y\\
\dot{y}=3a(1+a)x^2+(15(1+a)-2(9+5a)x)y+16y^2
\end{array}%
\right.
\]

 which possesses the irreducible invariant algebraic curve of
 degree four and genus one
$$g(x,y)\equiv{y^3+\displaystyle\frac{1}{4}(3(1+a)-6(1+a)x)y^2+\displaystyle\frac{3}{4}(1+a)ax^2y+\displaystyle\frac{3}{4}(1+a)a^2x^4}=0.$$

Another example we can find in the paper \cite{Ll1}.

This diferential system   \[\dot{x}=2(1+2x-2ax^2+6xy),\quad
\dot{y}=8-3a-14ax-2axy-8y^2\]a has the invariant algebraic curve
of degree four and genus one
\[x^2y^2+xy+ax^3-x^2+x+1/4=0\]

In the  both case we have one algebraic limit cycles.

In \cite{Ll1} we can find a surprisingly simple quadratic system

\[
\left\{%
\begin{array}{cc}
\dot{x}=xy+x^2+1\\
\dot{y}=3y^2-\displaystyle\frac{81}{2}x^2+\displaystyle\frac{57}{2}
\end{array}%
\right.
\]
which has invariant algebraic curve of degree 12 and genus one

\[
\left\{%
\begin{array}{cc}
-442368-7246584x^2+71546517x^4-97906500x^6+41343750x^8-23437500x^{10}+\\
48828125x^{12}+(322272x-12126312x^3+23463000x^5+1125000x^7+15625000x^9)y-
\\(98784-711288x^2+5058000x^4-375000x^6)y^2+(32928x-1124000x^3)y^3-5488y^4=0
\end{array}%
\right.
\]

The main result of this paper is to construct a set of quadratic
planar vector field with an invariant curves of large degree and
genus greater o equal to one. We apply the normal form of the
quadratic fields deduced by J.Llibre \cite{Ll2}.

\section{Quadratic  vector fields with a given  invariant algebraic curves of degree nine}

In this section we determine three quadratic vector field with
invariant curves of degree nine and genus one.

\begin{corollary}
The following quadratic differential systems have the invariant
curve $g=0$ of degree nine and genus one and cofactor $K=9y.$
\begin{itemize}
\item[(i)]
\[\dot{z}=zy+1,\quad
\dot{y}=3y^2-\frac{8}{13}qzy-\frac{24}{169}q^2z^2+q\]

\[\begin{array}{cc}
g=y^3-\frac{9}{26}q^2z-\frac{2784}{28561}z^5q^4+\frac{296}{2197}z^3q^3+\frac{6144}{371293}z^7q^5-\frac{4096}{
4826809}z^9q^6+\frac{9}{26}yq-\frac{60}{169}yq^2z^2\\
\\
+\frac{960}{2197}yz^4q^3-\frac{1536}{28561}yz^6q^4-\frac{12}{13}y^2qz-\frac{120}{169}y^2z^3q^2=0
\end{array}%
\]

\item[(ii)]
\[\dot{z}=zy+1,\quad
\dot{y}=3y^2-\frac{14}{25}qzy-\frac{84}{125}q^2z^2+q\]

\[\begin{array}{cc}
g=y^3+\frac{772}{3125}z^3q^3-\frac{39648}{390625}z^5q^4-\frac{9}{25}q^2z+\frac{9}{25}yq-\frac{312}{625}yq^2z^2+\\
\\
\frac{2688}{15625}yz^4q^3-\frac{9408}{78125}yz^6q^4-\frac{21}{25}y^2qz-\frac{192}{625}y^2z^3q^2-\frac{87808}{
9765625}z^9q^6+\frac{65856}{1953125}z^7q^5=0\end{array}\]

\item[(iii)]

\[\dot{z}=zy+1,\quad
\dot{y}=3y^2-\frac{88}{53}qzy-\frac{3264}{2809}q^2z^2+q\]

\[\begin{array}{cc}g=y^3-\frac{9}{106}q^2z+\frac{7655424}{7890481}z^5q^4-\frac{190316544}{418195493}z^7q^5-
\frac{105344}{148877}z^3q^3-\frac{132}{53}y^2qz+\frac{3840}{2809}y^2z^3q^2+\\
\\
\frac{1568669696}{22164361129}z^9q^6+\frac{9}{106}yq+
\frac{6240}{2809}yq^2z^2
-\frac{337920}{148877}yz^4q^3+\frac{4325376}{7890481} yz^6q^4=0
\end{array}\]

\end{itemize}
\end{corollary}

\section{Quadratic  vector fields with a given  invariant algebraic curves of degree twelve}
In this section we determine five quadratic vector field with
invariant curves of degree twelve and genus one.

\begin{corollary}
The following quadratic differential systems have the invariant
curve $g=0$of degree nine and genus one with cofactor $K=12y$.
\begin{itemize}
\item[(i)]
\[\dot{z}=zy+1,\quad\dot{y}=3y^2+pzy+(-pq-2p^2)z^2+q\]

\[\begin{array}{cc}
g=\frac{1}{36}p^6(q+2p^2)z^{12}+\frac{1}{2}p(\frac{1}{3}p^5q+\frac{2}{3}p^6)z^{10}+
(\frac{1}{3}p^5q+\frac{2}{3}q^6_{11})yz^9+\\
\\
(\frac{1}{2}(\frac{1}{3}qp^3+\frac{5}{3}p^4)(-q_{11}q-2p^2)+\frac{1}{4}p^2(\frac{1}{3}qp^3+\frac{5}{3}p^4)+3/4p^5q+3/2p^6)z^8+\\
\\
p(\frac{1}{3}qp^3+\frac{5}{3}p^4)yz^7+(\frac{17}{12}p(\frac{1}{3}qp^3+\frac{5}{3}p^4)+(\frac{1}{3}qp^3+\frac{5}{3}p^4)y^2+\\
\\
\frac{1}{4}p^5+\frac{1}{3}(\frac{1}{3}qp^3+\frac{5}{3}p^4)q+\frac{5}{4}p^3(-pq-2p^2))z^6+\\
\\
(4p^4+\frac{1}{2}qq^3+\frac{3}{2}(-pq-2p^2)p^2)yz^5+\\
\\
(\frac{19}{16}qp^3+\frac{23}{8}p^4+\frac{89}{48}(-pq-2p^2)p^2+\frac{1}{4}(-pq-2p^2)^2+3p^3y^2)z^4\\
\\
+((qp^2+3p^3+\frac{4}{3}p(-pq-2p^2))y+2p^2y^3)z^3+\\
\\
(\frac{119}{120}qp^2+\frac{59}{120}p(-pq-2p^2)+\frac{59}{60}p^3+(-pq+p^2)y^2+\frac{1}{3}
(-q_{11}q_{20}-2q^2_{11})q_{20})z^2\\
\\
+(2y^3p+(\frac{2}{3}pq+\frac{1}{3}p^2)y)z+y^4+(\frac{2}{3}q+\frac{1}{3}p)y^2+\frac{1}{9}q^2+\frac{1}
{36}p^2+\frac{1}{9}pq=0
\end{array}\]
\item[(ii)]
\[\dot{z}=zy+1,\quad\dot{y}=3y^2-\frac{35}{19}qzy-\frac{525}{361}q^2z^2+q\]

\[\begin{array}{cc}
g=\frac{16748046875}{271737008656}z^{12}q^8-\frac{3349609375}{7150973912}q^7z^{10}+\frac{95703125}{188183524}q^6yz^9\\
\\
+\frac{1003515625}{752734096}q^6z^8-\frac{7109375}{2476099}q^5yz^7+(\frac{203125}{130321}q^4y^2-\frac{16765625}{9904396}q^5)z^6\\
\\
+\frac{1396875}{260642}q^4yz^5+(\frac{1638125}{2085136}q^4-
\frac{39375}{6859}q^3y^2)z^4+(\frac{750}{361}q^2y^3-\frac{22625}{6859}q^3y)z^3+\\
\\
(\frac{155}{2888}q^3+\frac{1875}{361}q^2y^2)z^2+(\frac{-70}{19}qy^3-
\frac{155}{1444}q^2y)z-\frac{1}{5776}q^2+y^4+\frac{1}{19}qy^2=0
\end{array}\]
\item[(iii)]
\[\dot{z}=zy+1,\quad\dot{y}=3y^2-\frac{10}{59}qzy-\frac{150}{3481}z^2q^2+q\]
\[\begin{array}{cc}
g=
-\frac{781250000}{50362840098282103}q^8z^{12}+\frac{312500000}{853607459292917}q^7z^{10}-q^6_{20}\frac{62500000}{14467923038863}yz^9-\\
\\
q^6\frac{159375000}{14467923038863}z^8-q^5\frac{3750000}{245219034557}yz^7+(q^4\frac{375000}{4156254823}y^2+
\frac{230375000}{245219034557}q^5)z^6-\\
\\
q^4_{20}\frac{1425000}{593750689}yz^5+(q^3\frac{-4215000}{70444997}y^2-\frac{82923125}{4156254823}q^4)z^4+\\
\\
(\frac{5291500}{70444997}yq^3+\frac{281000}{1193983}y^3q^2)z^3+(q^2\frac{210750}{1193983}y^2+\frac{64980}{1193983}q^3)z^2+\\
\\
(q\frac{-20}{59}y^3-q^2\frac{129960}{1193983}y)z+\frac{110592}{1193983}q^2+y^4+q\frac{36}{59}y^2=0
\end{array}\]
\item[(iv)]
\[\dot{z}=zy+1,\quad\dot{y}=3y^2-\frac{55}{311}qzy+\frac{363}{96721}q^2z^2+q\]
\[\begin{array}{cc}
g=\frac{78460709418025}{22403871730541898055936}q^8z^{12}+\frac{35663958826375}{36019086383507874688}q^7z^{10}-
\frac{648435615025}{57908499008855104}yz^9q^6
-\\
\\
\frac{10221703240485}{231633996035420416}q^6z^8+\frac{16076916075}{46550240360816}yq^5z^7+(-\frac{292307565}{149679229456}q^4y^2
-\frac{69082021195}{186200961443264}q^5)z^6+\\
\\
\frac{611188545}{299358458912}q^4_{20}yz^5+(-\frac{29164418129}{2394867671296}q^4-\frac{12078825}{240641848}q^3y^2)z^4+\\
\\
(\frac{73205}{386884}q^2y^3+\frac{23645215}{481283696}q^3y)z^3+(\frac{298023}{1547536}q^2y^2+\frac{689337}{12380288}
q^3)z^2+\\
\\
(-\frac{110}{311}qy^3-\frac{689337}{6190144}yq^2)z+\frac{189}{311}y^2q+\frac{2278125}{24760576}q^2+y^4=0
\end{array}\]
\item[(iv)]
\[\dot{z}=zy+1,\quad\dot{y}=3y^2-\frac{187}{107}qzy-\frac{8349}{11449}q^2z^2+q\]
\[\begin{array}{cc}
g=\frac{1227125495297911}{274909788773107216}q^8z^12-\frac{82455072806579}{1284625181182744}q^7z^10+\\
\\
\frac{440936218217}{6002921407396}yz^9q^6+\frac{8241027180045}{24011685629584}z^8q^6-\frac{10932302931}{14025517307}yq^5
z^7+\\
\\
(-\frac{45176577061}{56102069228}q^5+\frac{58461513}{131079601}y^2q^4)z^6+\frac{706852839}{262159202}yq^4z^5+\\
\\
(\frac{1471610833}{2097273616}q^4-\frac{3733455}{1225043}y^2q^3)z^4+(\frac{13310}{11449}y^3q^2-\frac{3587045}{1225043}
yq^3)z^3+\\
\\
(\frac{54087}{11449}q^2_{20}y^2+\frac{7623}{91592}q^3)z^2+(-\frac{7623}{45796}yq^2-\frac{374}{107}y^3q)z+\frac{27}{
183184}q^2+y^4+\frac{9}{107}qy^2=0
\end{array}\]

\item[(v)]
\[\dot{z}=zy+1,\quad\dot{y}=3y^2+qzy-q^2z^2-q\]
\[\begin{array}{cc}
(\frac{1}{64}q^6p-\frac{1}{288}q^8)z^{12}+(\frac{3}{32}q^5p-\frac{1}{48}q^7)z^{10}+(-\frac{1}{24}yq^6+\frac{3}{16}
yq^4p)z^9+\\
\\
(-\frac{5}{96}q^6+\frac{15}{64}q^4p)z^8+(\frac{3}{4}yq^3p-\frac{1}{6}yq^5)z^7+(\frac{3}{4}y^2q^2p-
\frac{13}{72}q^5+\frac{5}{16}q^3p-\frac{1}{6}y^2q^4)z^6+\\
\\
(-\frac{1}{4}q^4+\frac{9}{8}q^2p)yz^5+(-\frac{37}{96}q^4+\frac{15}{64}q^2p+\frac{3}{2}y^2qp)z^4
+(-\frac{5}{6}yq^3+y^3p+\frac{3}{4}yqp)z^3+\\
\\
(\frac{3}{32}qp+\frac{3}{4}y^2p-\frac{17}{48}q^3+\frac{1}{2}y^2q^2)z^2+(2y^3q+\frac{3}{16}yp-\frac{17}{24}y
q^2)z+\frac{1}{64}p+y^4-\frac{1}{3}y^2q-\frac{1}{288}q^2=0
\\
\end{array}\]

\end{itemize}
\end{corollary}

\begin{proposition}
The quadratic vector fields with invariant algebraic curve of
genus one admits at most two algebraic limit cycles.
\end{proposition}
\begin{proof} The proof follow from the fact that if the algebraic curve has genus
$\textsc{G}$ then admits at most $\textsc{G}+1$ ovals.
\end{proof}

\section{Quadratic  vector fields with a given  invariant algebraic curves of degree fifteen}
We construct three quadratic vector field with invariant curve of
degree fifteen and genus two. In this section we give two of
theses quadratic vector fields.

\begin{corollary}
The following quadratic differential systems have the invariant
curve $g=0$ of degree fifteen  with cofactor $K=15y$ and genus
two.
\begin{itemize}
\item[(i)]
\[\dot{z}=zy+1,\quad\dot{y}=3y^2-6/17qzy-8/289q^2z^2+q\]

\[\begin{array}{cc}
g=y^5-135/1156q^3z+204600/410338673z^9q^7-49546/1419857z^5q^5-\\
\\
31800/24137569z^7q^6-365760/6975757441q^8z^{11}-36864/2015993900449z^15q^{10}+\\
\\
11915/250563z^3q^4-300/289z^3y^4q^2-15/17zy^4q-69120/6975757441yz^{12}q^8+\\
\\
14250/83521yz^4q^4-99000/24137569yz^8q^6-685/4913yq^3z^2+19800/1419857yz^6q^5+\\
\\
190080/410338673yq^7z^{10}-15780/83521y^2z^5q^4+11880/1419857y^2q^5z^7-190/289y^2q^2z-\\
\\
485/4913y^2q^3z^3-31680/24137569y^2z^9q^6-145/289y^3q^2z^2+3600/4913y^3z^4q^3-\\
\\
1320/83521y^3z^6q^4+207360/118587876497q^9z^13+135/1156yq^2+35/51y^3q=0
\end{array}{cc}\]

\item[(ii)]
\[\dot{z}=zy+1,\quad\dot{y}=3y^2-4/3pzy-16/3p^2z^2+p\]

\[\begin{array}{cc}
g=y^5+2621440/19683yz^{12}p^8+2560000/19683yp^6z^8-313600/6561yp^5z^6-\\
\\
3768320/19683yp^7z^{10}+1300/81y^4z^3p^2-10/3y^4zp+265/27y^3p^2z^2-\\
\\
10400/243y^3z^4p^3+474880/6561y^3z^6p^4+143840/2187y^2p^4z^5-949760/6561y^2p^5z^7\\
\\
+35/54y^2p^2z-9130/729y^2z^3p^3+942080/6561y^2z^9p^6-695/486yp^3z^2+\\
\\
7000/729y^4z^4+5/18y^3p+8388608/177147z^{15}p^{10}-10/81p^3z+1015/1458p^4z^3-\\
\\
2512/729p^5z^5+10/81yp^2+92800/6561p^6z^7-7577600/177147p^7z^9+\\
\\
4751360/59049z^11p^8-5242880/59049p^9z^{13}=0
\end{array}{cc}\]

\end{itemize}
\end{corollary}
\section{Quadratic  vector fields with a given  invariant algebraic curve of degree eighteen}
We construct a quadratic vector field with invariant curve of
degree eighteen and genus two.

\begin{corollary}
The following quadratic differential systems have the invariant
curve $g=0$ of degree  eighteen with cofactor $K=18y$ and genus
two.
\[\dot{z}=zy+1,\quad\dot{y}=3y^2-6/17qzy-8/289q^2z^2+q\]

\[\begin{array}{cc}
g=\frac{16777216}{23298085122481}z^{18}q^{12}-\frac{50331648}{1792160394037}q^{11}z^{16}+\frac{12582912}{137858491849}yz^{15}q^{10}+\\
\\
\frac{60555264}{137858491849}z^{14}q^{10}-\frac{26738688}{10604499373}yq^9z^{13}+\\
\\
(-\frac{2818048}{815730721}q^9+\frac{3342336}{815730721}y^2q^8)z^{12}+\frac{20840448}{815730721}yz^{11}q^8+\\
\\
(\frac{11452416}{815730721}q^8-\frac{1728}{4826809}q^6p-\frac{4325376}{62748517}y^2q^7)z^{10}+\\
\\
(-\frac{540672}{4826809}yq^7+\frac{360448}{4826809}y^3q^6)z^9+\\
\\
(\frac{2592}{371293}q^5p+\frac{1695744}{4826809}y^2q^6-\frac{1744896}{62748517}q^7+\frac{288}{28561}y^2q^4p)z^8+\\
\\
(-\frac{196608}{371293}y^3q^5+\frac{964608}{4826809}yq^6-\frac{720}{28561}yq^4p-\frac{64}{2197}y^3q^3p)z^7+\\
\\
(\frac{4}{169}y^4pq^2-\frac{196608}{371293}y^2q^5+\frac{132544}{4826809}q^6-\frac{90}{2197}q^4p-\frac{264}{2197}y^2q^3p+\frac{12288}{
28561}y^4q^4)z^6+\\
\\
(\frac{474}{2197}yq^3p+\frac{18432}{28561}y^3q^4-\frac{59136}{371293}yq^5+\frac{136}{169}y^3pq^2)z^5+\\
\\
(-\frac{5712}{371293}q^5+\frac{241}{4394}pq^3+\frac{10896}{28561}y^2q^4-\frac{69}{338}y^2pq^2-\frac{20}{13}y^4qp)z^4+\\
\\
(-\frac{1}{2}y^3pq-\frac{608}{2197}y^3q^3+\frac{1104}{28561}yq^4+y^5p-\frac{11}{52}ypq^2)z^3+\\
\\
(\frac{177}{2197}q^4+\frac{3}{26}y^2pq+\frac{204}{169}y^4q^2+\frac{1224}{2197}y^2q^3+\frac{3}{4}y^4p-\frac{23}{832}pq^2)z^2+\\
\\
(\frac{3}{16}y^3p+\frac{23}{416}ypq-\frac{354}{2197}yq^3-\frac{180}{169}y^3q^2-\frac{24}{13}y^5q)z+y^6+\\
\\
\frac{15}{2197}q^3+\frac{9}{13}y^4q+\frac{1}{208}pq+\frac{1}{64}y^2p+\frac{24}{169}y^2q^2=0
\end{array}{cc}\]
\end{corollary}

\end{document}